\documentclass[12pt]{amsart}
\pdfoutput=1
\usepackage{fullpage}
\usepackage{amssymb}
\usepackage{amstext}
\usepackage{amsmath}
\usepackage{amsthm}
\usepackage{stmaryrd}
\usepackage{verbatim}
\usepackage[colorlinks,linkcolor=blue]{hyperref}

\newtheorem{theorem}{Theorem}[section]

\newtheorem{lemma}[theorem]{Lemma}
\theoremstyle{remark}

\begin{document}
\title{Shimura lift of Rankin-Cohen brackets of eigenforms and theta series}
\author{Wei Wang}
\address{Department of Mathematics, Shaoxing University, Shaoxing 312000, China}
\email{weiwang\_math@163.com}

\begin{abstract}
The Shimura lift of a Hekce eigenform multiplied by a theta series is the square of the form. We extend this result by generalizing the product map to the Rankin-Cohen bracket. We prove that the Shimura lift of Rankin-Cohen bracket of an eigenform and a theta series is given by  Rankin-Cohen bracket of the eigenform and itself.
\end{abstract}
\maketitle

\section{Introduction}
The  half-integral weight modular forms were first systematically studied by Shimura in \cite{MR0332663}. Shimura gave a correspondence between certain modular forms of even weight $2k$ and modular forms of half-integral weight $k+1/2$. To describe it, we fix some standard notation. For $t\in\mathbb{Z}$, denote by $\chi_t$ the primitive character corresponding to the field extension $\mathbb{Q}(t^{1/2})/\mathbb{Q}$. Suppose that $f(z)=\sum_{n=1}^{\infty}a(n)q^n\in S_{k+1/2}(4N,\chi)$ is a  half-integral weight modular form, where $k$ is a positive integer. Let $t$ be a square-free positive integer. Define complex numbers $A_t(n)$ by
\[
\sum_{n\geq 1}A_t(n)n^{-s}:=L(s-k+1,\chi\chi^k_{-4}\chi_t)\sum_{n\geq 1}a(tn^2)n^{-s},
\]
then the function
\[
S_t(g(z)):=\sum_{n\geq 1}A_t(n)q^n
\]
is in $M_{2k}(2N,\chi^2)$.  Furthermore, $S_t(g(z))$ is a cusp form if $k\geq 2$. The map $S_t$ is called the Shimura lift.

A typical example of half-integral weight modular forms is given by the theta series. Let $\psi$ be a Dirichlet character modulo $r$ and let $\psi(-1)=(-1)^v$, $v=0$ or 1 according as $\psi$ is even or odd. For a positive integer $d$, define 
\[
\theta_{\psi}(dz)=\sum_{n\in\mathbb{Z}}n^v\psi(n)q^{dn^2}.
\]
We have $\theta_{\psi}(dz)\in M_{v+1/2}(4r^2d,\psi\chi_{d}\chi_{-4}^v)$, moreover, if $\psi(-1)=-1$, $\theta_{\psi}(dz)$ is a cusp form. By the graded algebraic structure of modular forms, for a modular form $f$ of integral weight, the product of $f$ and a theta series is a half-integral weight modular form. Another way to produce a new half-integral weight modular form is using Rankin-Cohen brackets. Let $f$ and $g$ be two smooth functions and let $k$ and $l$ be two integers or half-integers. The $w$-th Rankin–Cohen bracket for $f$ and $g$ is defined as
\[
[f,g]_{w}:=(2\pi i)^{-w}\sum_{0\leq j\leq w} (-1)^j \binom{k+w-1}{w-j} \binom{l+w-1}{j}\frac{d^jf}{dz^j} \frac{d^{w-j}g}{dz^{w-j}}.
\]
The 0-th Rankin–Cohen bracket of two modular forms is the product of these
two modular forms. The following theorem given by Cohen \cite{MR0382192} shows that Rankin-Cohen brackets in a sense is a generalization of the product map. 
\begin{theorem}[Cohen]
Let $f$, $g$ be two modular forms of weights $k$, $l$ and characters 
$\chi_1$, $\chi_2$ respectively, on a subgroup $\Gamma$ of $\operatorname{SL}_2(\mathbb{Z})$. Then $[f,g]_w\in M_{k+l+2w}(\Gamma, \chi_1\chi_2\chi)$
where $\chi=1$ if $k_1$ and $k_2\in \mathbb{Z}$, $\chi=\chi_{-4}^{k_i}$ if $k_i\in \mathbb{Z}$ and $k_{3-i}\in 1/2+\mathbb{Z}$, $\chi=\chi_{-4}^{k_1+k_2}$ if $k_1$ and $k_2\in 1/2+\mathbb{Z}$. Furthermore, $[f,g]_w$ is a cusp form if $w>0$. 
\end{theorem}

The main result of this paper is to describe the precise image of the Shimura lift of Rankin-Cohen brackets of a normalized eigenform and a theta series. Recall that a normalized eigenform $f\in M_k(N,\chi)$ is an eigenfunction of all the  Hecke operators. In other words, its Fourier coefficient satisfy that $a(1)=1$ and for all positive integers $m,n$,
\[
a(m)a(n)=\sum_{d\mid \gcd(m,n)}\chi(d)d^{k-1}a(mn/d^2),
\]
or equivalently,
\begin{equation}\label{equation_newform}
a(mn)=\sum_{d\mid \gcd(m,n)}\mu(d)\chi(d)d^{k-1}a(m/d)a(n/d).
\end{equation}
Non-cuspidal Hecke eigenforms are Eisenstein series, and their Fourier coefficients are easy to obtain. The cuspidal Hecke eigenforms are called newforms. Unlike  Eisenstein series, we know little about Fourier coefficients of newforms.

For a modular form $f$ with Fourier expansion $\sum_{n\geq 0}a(n)q^n$ and a Dirichlet character $\psi$, the $\psi$-twist of $f$ is defined as
\[
f_{\psi}(z):=\sum_{n\geq 0}\psi(n)a(n)q^n.
\]
Our main result is as follows.
\begin{theorem}\label{Main_theorem}
Let $f=\sum_{n\geq 0}a(n)q^n\in M_k(N,\chi)$ be a normalized Hecke eigenform. Let $\psi$ be a Dirichlet character modulo $r$. Write $\prod_{i=1}^l\psi_{p_i^{\alpha_i}}$ be the decomposition of $\psi$. For divisor $d$ of $r$ such that $\gcd(d,r/d)=1$, we write $\psi_d=\prod_{p_i\mid d}\psi_{p_i^{\alpha_i}}$. Let $t$ be a square-free positive integer such that $t\mid N$ and $\gcd(t,r)=1$. The $t$-Shimura lift of the modular form $\left[f(4rz),\theta_{\psi}(tz)\right]_w\in S_{k+2w+v+1/2}(4N^{'}r^2,\psi\chi\chi_t\chi_{-4}^{k+v})$, where $N^{'}=N/\gcd(N,r)$, is
\[
\binom{2w+v}{w}\binom{k+2w+v-1}{k+w-1}^{-1}a(t)t^w(g_{\psi}(z)-\psi(2)\chi(2)2^{k+2w+v-1}g_{\psi}(2z)),
\]
where the modular form $g(z)$ is given by
\[
\sum_{\substack{d\mid r \\ \gcd(d,r/d)=1}}\psi_{d}(-1)\left[f(rz/d),f(dz)\right]_{2w+v}.
\]
\end{theorem}
We illustrate this theorem with an explicit example. Let $f(z)=\eta^4(z)\eta^4(5z)$ be the newform in $S_4(5)$ with $q$-expansion
\[
f=q - 4q^2 + 2q^3 + 8 q^4 - 5q^5 - 8q^6 + 6q^7 - 23q^9 + O(q^{10}).
\]
Take $\psi=\left(\frac{12}{\cdot}\right)$ to be the primitive even Dirichlet character modulo 12, then $\theta_{\psi}(z)=2\eta(24z)\in M_{1/2}(576,\psi)$, and let $g$ be $2[f(z),f(12z)]_2-2[f(3z),f(4z)]_2$ with $q$-expansion
\[
g=100q^7 - 640q^{10} + 1040q^{11} + 2020q^{13} - 640q^{14} - 7500q^{17} - 16140q^{19} +O(q^{20}).
\]
Theorem \ref{Main_theorem} tells us that the Shimura lift of 
$[f(48z),\eta(24z)]_1\in S_{6+1/2}(2880,\psi)$ is $g_{\psi}/5$.

In the following, we focus our attention on a subspace of half-integral weight modular forms, namely Kohnen's plus space. Let $N$ be a positive odd integer and $\chi$ a Dirichlet character modulo $N$ with $\chi(-1)=(-1)^v$. The Kohnen plus space, denoted $S_{k+1/2}^{+}(N,\chi)$, is the subspace of $S_{k+1/2}(4N,\chi\chi_{-4}^v)$ consisting of all cusp forms $f$ whose $n$-th Fourier coefficient vanishes unless $(-1)^{k+v}n\equiv 0,1\mod 4$. Suppose that $f(z)=\sum_{n\geq 1}a(n)q^n\in S^{+}_{k+1/2}(N,\chi)$. Let $D$ be a fundamental discriminant with $(-1)^{k+v}D>0$. The Shimura lift 
\[
S_{D}^{+}(f(z)):=\sum_{n\geq 1}A_{D}(n)q^n
\]  
is defined by the Dirichlet series
\[
a(|D|)\sum_{n\geq 1}A_{D}(n)n^{-s}=L(s-k+1,\chi\chi_{D})\sum_{n\geq 1}a(|D|n^2)n^{-s},
\]
then $S_{D}^{+}(f(z))\in S_{2k}(N,\chi^2)$. When $N$ is odd squarefree and $\chi$ is quadratic, Kohnen showed there is a newform theory attached to the plus space $S^{+}_{k+1/2}(N,\chi)$ and there is a linear combination of the Shimura lift  $S_{D}^{+}$ which maps the space of newforms of weight $k+1/2$ isomorphically onto the space of newforms of weight $2k$ for $\Gamma_0(N)$. For more details, see \cite{MR0660784}. Considering Kohnen's plus space is quite natural, because Rankin-Cohen brackets of eigenforms and theta series fall within this subspace. Our theorem in the sense of Kohnen's plus space is as follows.
\begin{theorem}\label{theorem_plus}
Let $f=\sum_{n\geq 0}a(n)q^n\in M_k(N,\chi)$ be a normalized Hecke eigenform. Let $\psi$ be a Dirichlet character modulo $r$ with $\psi(-1)=(-1)^v$. Let $D$ be a fundamental discriminant such that $D>0$, $D\mid N$ and $\gcd(D,r)=1$. The $D$-Shimura lift $S_{D}^{+}$ of the modular form $\left[f(4rz),\theta_{\psi}(Dz)\right]_w\in S^{+}_{k+2w+v+1/2}(N^{'}r^2,\psi\chi\chi_{D})$, where $N^{'}=N/\gcd(N,r)$, is
\[
\binom{2w+v}{w}\binom{k+2w+v-1}{k+w-1}^{-1}D^wg_{\psi}(z),
\]
where the modular form $g(z)$ is given by
\[
\sum_{\substack{d\mid r \\ \gcd(d,r/d)=1}}\psi_{d}(-1)\left[f(rz/d),f(dz)\right]_{2w+v}.
\]
\end{theorem}
Selberg in an unpublished work showed that when $f\in S_k(\operatorname{SL}_2(\mathbb{Z}))$ is a normalized Hecke eigenform, then the image of the function $f(4z)\theta(z)$ under $S_1$ is $f^2(z)-2^{k-1}f^2(2z)$. Kohnen and Zagier \cite{MR0629468} extended it to $S^{+}_{D}$ for certain Eisenstein series, and this is the most crucial step in the proof of their main theorem. Cipra generalized Selberg's work in \cite{MR1002114}, and he gave the precise image of the Shimura lift of the product of a newform of level $N$ and a theta series attached to a primitive Dirichlet character modulo a prime power. It is worth noting that Rankin-Cohen brackets has already appeared in Cipra's theorem when the character is odd. Later on, Hansen and Naqvi in \cite{MR2456838} generalized Cipra's result by removing the condition that the modulus of the character is a prime power. In 2021, Pandey and Ramakrishnan \cite{MR4287581} generalized the above results to $t$-Shimura lift when $t\neq 1$ is a square-free positive integer. A special case of Theorem \ref{theorem_plus} has already been provided by Popa in \cite{Popa2011RationalDO}. Choie, Kohnen, and Zhang \cite{choie2024rankincohenbracketsheckeeigenforms} corrected Popa's constant and used this result to obtain formulas related Rankin-Cohen brackets of normalized
Hecke eigenforms and certain shifted convolution series.

The second result of this paper mainly focuses on the case where $\psi$ is a trivial character. More general result can be obtained in this case. To state this result, we make some notational conventions. For a positive integer $d$, let $\chi_{0,d}$ be the trivial character modulo $d$. For a modular form $f$ and $a,b\in\mathbb{R}$, $l\in\mathbb{N}$, define
\[
f(z)\big|(aI+bB(l)):=af(z)+bf(lz).
\]
Let $f(z)=\sum_{n\geq 0}a(n)q^n\in M_k(N,\chi)$ be a normalized Hecke eigenform and $r$ be a positive integer. Let $t$ be a square-free positive integer such that $t\mid N$ and $\gcd(t,r)=1$. Then $\left[f(4rz),\theta(tz)\right]_{w}$ is in $M_{k+2w+1/2}(4Nr,\chi\chi_t\chi_{-4}^{k})$. Moreover, this modular form can be placed in the space $M_{k+2w+1/2}(M,\chi\chi_t\chi_{-4}^{k})$, for any positive integer $M$ such that $4Nr\mid M$. The following theorem  provides the expression for the Shimura lift of $\left[f(4rz),\theta(tz)\right]_{w}$ in the space $S_{k+2w+1/2}(M,\chi\chi_t\chi_{-4}^{k})$. For integer $n$, let $\operatorname{rad}(n)$ denote the radical of $n$ defined as the product of the distinct prime numbers dividing $n$. 
\begin{theorem}\label{Main_theorem2}
With the above notation and assumptions. The $t$-Shimura lift of the cusp form $\left[f(4rz),\theta(tz)\right]_w\in S_{k+2w+1/2}(M,\chi\chi_t\chi_{-4}^{k})$ is
\[
\binom{2w}{w}\binom{k+2w-1}{k+w-1}^{-1}a(t)t^w\sum_{d\mid r}[f(rz/d),f_{\chi_{0,r/d}}(dz)]_{2w}\bigg|\prod_{p\mid M_d}(I-\chi(p)p^{k+2w-1}B(p)),
\]
where $M_d:=2\operatorname{rad}(M)/\operatorname{rad}(r/d)$.
\end{theorem}
Regarding this result, we also illustrate it with an example. For prime $p\Vert N$, let $f\in S_k(N/p)$ be a newform. By Theorem \ref{Main_theorem2}, the Shimura lift of the modular form $f(4pz)\theta(z)\in S_{k+1/2}(4N)$ is 
\[
\left(f(pz)f_{\chi_{0,p}}(z)+f(z)f(pz)-p^{k-1}f(pz)f(p^2z)\right)\big|(I-\chi_{0,N}(2)2^{k-1}B(2)).
\]
Since $f$ is an eigenform of the Hecke operator $T_p$, $f_{\chi_{0,p}}(z)=f(z)-a_p(f)f(pz)+p^{k-1}f(p^2z)$, where $a_p(f)$ denotes the $p$-th Fourier coefficient of $f$. Therefore the Shimura lift of $f(4pz)\theta(z)$ is 
\[2f(z)f(pz)-a_p(f)f^2(pz)-\chi_{0,N}(2)2^{k-1}(2f(2z)f(2pz)-a_p(f)f^2(2pz)).\]

By the theory of newforms, the space $S_{k}(M,\chi)$ can be decomposed into direct sums of newforms and oldforms:
\[
S_{k}(M,\chi)=\bigoplus_{f(\chi)\mid N\mid M}\bigoplus_{d\mid M/N}S_{k}^{\text{new}}(N,\chi)\big|B(d),
\]
here $f(\chi)$ denotes the conductor of $\chi$. Once the spaces of newforms in the above decomposition can be obtained, Theorem \ref{Main_theorem2} actually provides a complete description of the Shimura lift of half-integral weight modular forms of the form $[S_{k}(M,\chi)\big|B(4),\theta]_w$. The space of Eisenstein series also has a corresponding decomposition, so the result can be extended to the modular form space.
Similarly, Theorem \ref{Main_theorem2} also has a version for the Kohnen's plus space, which can be obtained by dividing $M_d$ by 2 in Theorem \ref{Main_theorem2}. We will not describe it in detail.
\section{proof of theorems}
Before proceeding to prove the theorem, we establish the following lemma.
\begin{lemma}\label{lemma_comident}
Let $k,w$ be two positive integers and $v\in\left\{0,1\right\}$. We have the following combinatorial identity regarding variables $x,y$:
\begin{align*}
\sum_{j=0}^w&\frac{(k+2w-1+v)!}{j!(2w-2j+v)!(k+j-1)!}(xy)^j(x+y)^{2w-2j+v}\\
&=\sum_{j=0}^{2w+v}\binom{k+2w-1+v}{j}\binom{k+2w-1+v}{2w-j+v}x^jy^{2w-j+v}
\end{align*}
\end{lemma}
\begin{proof}
By the binomial theorem, we have
\begin{align*}
(1+xt)^{k+2w-1+v}(1+yt)^{k+2w-1+v}&=\sum_{i,j}\binom{k+2w-1+v}{j}\binom{k+2w-1+v}{i}x^jy^it^{i+j}\\
&=\sum_{n}\sum_{j=0}^n\binom{k+2w-1+v}{j}\binom{k+2w-1+v}{n-j}x^{n-j}y^it^{n}.
\end{align*}
On the other hand,
\begin{align*}
(1+xt)^{k+2w-1+v}&(1+yt)^{k+2w-1+v}=(1+(x+y)t+xyt^2)^{k+2w-1+v}\\
&=\sum_j\binom{k+2w-1+v}j(xy)^jt^{2j}(1+(x+y)t)^{k+2w-1+v-j}\\
&=\sum_{i,j}\binom{k+2w-1+v}j\binom{k+2w-1+v-j}{i}(xy)^j(x+y)^it^{i+2j}\\
&=\sum_{n}\sum_{j=0}^{\lfloor n/2\rfloor}\binom{k+2w-1+v}j\binom{k+2w-1+v-j}{n-2j}(xy)^j(x+y)^{n-2j}t^n.
\end{align*}
By comparing the coefficients of $t^{2w+v}$ in the above two equations and noting
\[
\binom{k+2w-1+v}j\binom{k+2w-1+v-j}{2w+v-2j}=\frac{(k+2w-1+v)!}{j!(2w-2j+v)!(k+j-1)!},
\]
the proof is complete.
\end{proof}

We now begin the proof of Theorem \ref{Main_theorem}. Denote by $b(n)$ the $n$-th Fourier coefficient of $\left[f(4rz),\theta_{\psi}(tz)\right]_w$. By definition of the Rankin-Cohen bracket, the $tn^2$-th Fourier coefficient $b(tn^2)$ is 
\begin{equation}\label{tn2-fouriercoe}
\sum_{j=0}^w(-1)^j\binom{k+w-1}{w-j}\binom{w-1/2+v}{j}\sum_{m\in \mathbb{Z}}\psi(m)t^{w-j}m^{2w-2j+v}(tn^2-tm^2)^ja\left(\frac{tn^2-tm^2}{4r}\right).
\end{equation}
Here and throughout the proof, we assume that $a(n)=0$ unless $n$ is a positive integer. Since the modulus of $\psi$ is $r$, if $\gcd(m,r)>1$ or $t(n-m)(n+m)\notin 4r\mathbb{N}$, the expression (\ref{tn2-fouriercoe}) in the second summation is equal to 0 and it is a finite sum. By our assumption, $\gcd(t,r)=1$, $tn^2$-th Fourier coefficient of $f$ is 0 unless $2r\mid (n-m)(n+m)$, hence $n+m$ and $n-m$ must both be even. And since $f$ is an eigenform and $t\mid N$, equation (\ref{equation_newform}) implies that $a(t(n^2-m^2)/4r)=a(t)a((n^2-m^2)/4r)$. Let $\gcd((n-m)/2,r)=d$, thus $n\equiv m\mod 2d$ and $n\equiv -m\mod 2r/d$. If $\gcd(d,r/d)=d^{'}>1$, this implies that $m\equiv n\equiv -m \mod 2d^{'}$, hence $\psi(m)=0$ and the corresponding sum is equal to 0. Therefore, we can assume that $\gcd(d,r/d)=1$. Since $m\equiv n\mod d$ and $m\equiv -n \mod r/d$, we have
\[
\psi(m)=\psi_d(m)\psi_{r/d}(m)=\psi_d(n)\psi_{r/d}(-n)=\psi_{r/d}(-1)\psi(n).
\]
Substituting $m=n-2dm^{'}$ for $m^{'}\in\mathbb{Z}$ and $m+n=2n-2dm^{'}$ in (\ref{tn2-fouriercoe}), we get
\begin{equation}\label{eqeuation_coefficient_btn2}
\begin{aligned}
b(tn^2)=t^wa(t)\psi(n)\sum_{\substack{d\mid r \\ \gcd(d,r/d)=1}}\psi_{r/d}(-1)\sum_{j=0}^w(-1)^j\binom{k+w-1}{w-j}\binom{w-1/2+v}{j}
\\
\times\sum_{m\in\mathbb{Z}}(4dm)^j(n-dm)^j(n-2dm)^{2w-2j+v}a\left(\frac{m(n-dm)}{r/d}\right).
\end{aligned}
\end{equation}
Now since $f$ is an eigenform with character $\chi$,  by (\ref{equation_newform}), we find that
\begin{equation}\label{eqeuation_coefficient_btn2_2}
\begin{aligned}
&\sum_{m\in\mathbb{Z}}(4dm)^j(n-dm)^j(n-2dm)^{2w-2j+v}a\left(\frac{m(n-dm)}{r/d}\right)\\
&=\sum_{m\in\mathbb{Z}}(4dm)^j(n-dm)^j(n-2dm)^{2w-2j+v}\sum_{\delta\mid\gcd(m,n)}\mu(\delta)\chi(\delta)\delta^{k-1}a\left(\frac{m}{\delta}\right)a\left(\frac{n-dm}{\delta r/d}\right)\\
&=\sum_{\delta\mid n}\mu(\delta)\chi(\delta)\delta^{k+2w+v-1}\sum_{m\in\mathbb{Z}}(4dm)^j(n/\delta-dm)^j(n/\delta-2dm)^{2w-2j+v}a(m)a\left(\frac{n/\delta-dm}{r/d}\right).
\end{aligned}
\end{equation}
Letting $x=-dm$ and $y=n-dm$ in the identity in Lemma \ref{lemma_comident} and using
\[
\binom{w-1/2+v}j=4^{-j}\frac{(2w+v)!(w-j)!}{w!j!(2w-2j+v)!},
\]
we get 
\begin{align*}
&\sum_{j=0}^w(-1)^j\binom{k+w-1}{w-j}\binom{w-1/2+v}{j}(4dm)^j(n-dm)^j(n-2dm)^{2w-2j+v}\\
&=\frac{(k+w-1)!(2w+v)!}{(k+2w-1+v)!w!}\sum_{j=0}^{2w+v}(-1)^j\binom{k+2w-1+v}{2w-j+v}\binom{k+2w-1+v}{j}(dm)^j(n-dm)^{2w-j+v}.
\end{align*}
Multiply both sides of the above identity by $a(m)a\left(\frac{n-dm}{r/d}\right)$ and sum over $m\in\mathbb{Z}$. After summing, up to a constant, the right-hand side of the resulting expression is indeed the $n$-th Fourier coefficient of $[f(dz),f(rz/d)]_{2w+v}$.  Then, replace $n$ by $n/\delta$, multiply both sides of the resulting expression by $\psi_{r/d}(-1)$ and $\mu(\delta)\chi(\delta)\delta^{k+2w+v-1}$, and sum over $d\mid r,\gcd(d,r/d)=1$ and $\delta\mid n$. After the above calculations, the left-hand side of the equation equals $b(tn^2)/(t^wa(t)\psi(n))$ by (\ref{eqeuation_coefficient_btn2}) and (\ref{eqeuation_coefficient_btn2_2}), and the right-hand side of the equation equals
\[
\frac{(k+w-1)!(2w+v)!}{(k+2w-1+v)!w!}\sum_{\delta\mid n}\mu(\delta)\chi(\delta)\delta^{k+2w+v-1}\tilde{c}(n/\delta),
\]
where $\tilde{c}(n)$ denotes the $n$-th Fourier coefficient of 
\[
\sum_{\substack{d\mid r \\ \gcd(d,r/d)=1}}\psi_{r/d}(-1)\left[f(dz),f(rz/d)\right]_{2w+v}.
\]
Denote the $n$-th Fourier coefficient of $g(z)$ as $c(n)$. Using $\psi_{r/d}(-1)\psi_d(-1)=(-1)^v$ and $[f(rz/d),f(dz)]_{2w+v}=(-1)^v[f(dz),f(rz/d)]_{2w+v}$ and summarizing the above assertions, we find that 
\begin{equation}\label{equation_btn2}
b(tn^2)=\frac{(k+w-1)!(2w+v)!}{(k+2w-1+v)!w!}t^wa(t)\sum_{\delta\mid n}\mu(\delta)\chi(\delta)\delta^{k+2w+v-1}\psi(\delta)\psi(n/\delta)c(n/\delta).
\end{equation}
For simplicity of notation, we write $C(t)$ instead of
\[
\frac{(k+w-1)!(2w+v)!}{(k+2w-1+v)!w!}t^wa(t)=\binom{2w+v}{w}\binom{k+2w+v-1}{k+w-1}^{-1}t^wa(t).
\]
Using the language of Dirichlet series, for $\operatorname{Re}(s)$ sufficiently large, (\ref{equation_btn2}) is equivalent to
\[
\sum_{n\geq 1}b(tn^2)n^{-s}=C(t)L(s-(k+2w+v)+1,\psi\chi)^{-1}\sum_{n\geq 1}\psi(n)c(n)n^{-s}.
\]
It follows from the definition of the Shimura lift that
\begin{equation}\label{equation_shimura}
\sum_{n\geq 1}A_t(n)n^{-s}=C(t)\frac{L(s-(k+2w+v)+1,(\psi\chi)^{'})}{L(s-(k+2w+v)+1,\psi\chi)}\sum_{n\geq 1}\psi(n)c(n)n^{-s},
\end{equation}
where $(\psi\chi)^{'}$ is the Dirichlet character modulo $4N^{'}r^2$ induced by $\psi\chi$. By Euler products of the corresponding $L$-functions, we have
\[
\frac{L(s-(k+2w+v)+1,(\psi\chi)^{'})}{L(s-(k+2w+v)+1,\psi\chi)}=(1-\psi(2)\chi(2)2^{k+2w+v-1-s}).
\]
Comparing the coefficients of both sides of the Dirichlet series (\ref{equation_shimura}) completes the proof.

The proof of Theorem \ref{theorem_plus} is completely identical to the above. The only difference is that, in definition of Kohnen's Shimura lift, the modulus of the character $\psi\chi$ in the Dirichlet series $L(s-(k+2w+v)+1,\psi\chi)$ is $N^{'}r^2$ rather than $4N^{'}r^2$.

The proof of Theorem \ref{Main_theorem2} follows along the same lines. We only discuss the case where $w=0$ and $t=1$, leaving the general case to the readers. Let $b(n)$ denote the $n$-th Fourier coefficient of $f(4rz)\theta(z)$, then
\begin{equation}\label{equation_proof_thm2}
b(n^2)=\sum_{m\in\mathbb{Z}}a\left(\frac{n^2-m^2}{4r}\right)=\sum_{m\in\mathbb{Z}}a\left(\frac{(n-m)(n+m)}{4r}\right).
\end{equation}
Similar to the proof of Theorem 1, we decompose the summation over $m\in\mathbb{Z}$ according to the greatest common divisor of $(n-m)/2$ and $r$. Let $\gcd ((n-m)/2,r)=d$ and substituting $m=n-2dm^{'}$ for $\gcd (m^{'},r/d)=1$ in (\ref{equation_proof_thm2}), we get
\begin{align*}
b(n^2)&=\sum_{d\mid r}\sum_{m\in\mathbb{Z}}\chi_{0,r/d}(m)a\left(\frac{m(n-dm)}{r/d}\right)\\
&=\sum_{d\mid r}\sum_{m\in\mathbb{Z}}\sum_{\delta\mid \gcd(n,m)}\mu(\delta)\chi(\delta)\delta^{k-1}\chi_{0,r/d}(m)a(m/\delta)a\left(\frac{n-dm}{\delta r/d}\right)\\
&=\sum_{d\mid r}\sum_{\delta\mid n}\mu(\delta)\chi(\delta)\chi_{0,r/d}(\delta)\delta^{k-1}\sum_{m\in\mathbb{Z}}\chi_{0,r/d}(m)a(m)a\left(\frac{n/\delta-dm}{r/d}\right).
\end{align*}
Denote the $n$-th Fourier coefficient of $f_{\chi_{0,r/d}}(dz)f(rz/d)$ as $c_{d}(n)$. The above calculation shows that 
\[
\sum_{n\geq 1}b(n^2)n^{-s}=\sum_{d\mid r}L(s-k+1,\chi\chi_{0,r/d})^{-1}\sum_{n\geq 1}c_d(n)n^{-s}.
\]
It follows from the definition of the Shimura lift that
\[
\sum_{n\geq 1}A_1(n)n^{-s}=\sum_{d\mid r}\frac{L(s-k+1,\chi\chi_{0,4M})}{L(s-k+1,\chi\chi_{0,r/d})}\sum_{n\geq 1}c_d(n)n^{-s}.
\]
By Euler products of the corresponding $L$-functions, we have
\[
\frac{L(s-k+1,\chi\chi_{0,4M})}{L(s-k+1,\chi\chi_{0,r/d})}=\prod_{p\mid 2\operatorname{rad}(M)/\operatorname{rad}(r/d)}(1-\chi(p)p^{k-1-s}).
\]
Comparing the coefficients of both sides of the Dirichlet series completes the proof.

We would like to finish this paper with a few remarks about our results. The modular form $\left[f(4rz),\theta_{\psi}(tz)\right]_w$ is always a cusp form unless $w=0$, both $f(z)$ and $\theta_{\psi}(tz)$ are not cusp forms. At this point, the Shimura lift is not defined. However, the calculations above still hold, and we can extend the definition of Shimura lift to non-cuspidal modular forms in this way. This is precisely what is done in \cite{MR1736567}.

One can verify directly that the form $g_{\psi}(z)$ in Theorem \ref{Main_theorem} and \ref{theorem_plus} is of level $N^{'}r^2$ and the form in Theorem \ref{Main_theorem2} is of level $2M$ without relying on Shimura and Niwa's theorem (cf. \cite{MR2456838}). When the plus space of half-integral weight $S^{+}_{k+1/2}(4N)$ is spanned by Rankin-Cohen brackets of eigenforms and theta series, these results provide another proof of the fact that the level of modular forms after lifting is $N$. For example, when $k\geq 6$ is an even integer, the set $\left\{[G_{k-2j}(4z),\theta(z)]_j,1\leq j\leq  k/2-2\right\}$ forms a basis of $S_{k+1/4}^{+}(4)$. In this case, our theorem provides an explicit expression for the Shimura lift.

The result of this paper may be generalized to other  automorphic forms. It is worth noting that in 1992, Shemanske and Walling \cite{ShimuSheWal} have obtained results similar to Cipra's theorem regarding Hilbert modular forms.

\providecommand{\bysame}{\leavevmode\hbox to3em{\hrulefill}\thinspace}
\providecommand{\MR}{\relax\ifhmode\unskip\space\fi MR }
\providecommand{\MRhref}[2]{%
  \href{http://www.ams.org/mathscinet-getitem?mr=#1}{#2}
}
\providecommand{\href}[2]{#2}

\end{document}